\documentclass[12pt]{amsart}

\usepackage{setspace}
\usepackage{amsmath, amssymb,amsthm, amsfonts, latexsym}

\newtheorem{theorem}{Theorem}
\newtheorem{lemma}{Lemma}
\newtheorem{prop}{Proposition}

\theoremstyle{remark}
\newtheorem{remark}{Remark}

\newcommand{\tr}{\mbox{tr}}

\renewcommand{\div}{\mbox{div}}

\newcommand{\R}{\mathbb R}

\newcommand{\be}{\begin{equation}}
\newcommand{\ee}{\end{equation}}
\newcommand{\bee}{\begin{equation*}}
\newcommand{\eee}{\end{equation*}}

\def\p{\partial}

\def\la{\langle}
\def\ra{\rangle}
\def\lf{\left}
\def\ri{\right}
\def\Pi{\displaystyle{\mathbb{II}}}

\def\bnabla{\overline{\nabla}}

\begin{document}

\title{Monotone quantities  involving a weighted $\sigma_k$ integral  along inverse curvature flows}

\author{Kwok-Kun Kwong}
\address[Kwok-Kun Kwong]{Department of Mathematics, University of Miami, Coral Gables, FL 33146, USA.}
 \email{kwong@@math.miami.edu}

\author{Pengzi Miao$^1$}
\address[Pengzi Miao]{Department of Mathematics, University of Miami, Coral Gables, FL 33146, USA.}
\email{pengzim@math.miami.edu}

\thanks{$^1$Research partially supported by Simons Foundation Collaboration Grant for Mathematicians \#281105.}


\renewcommand{\subjclassname}{
  \textup{2010} Mathematics Subject Classification}

\begin{abstract}
We give a family of monotone  quantities along smooth solutions to
the inverse curvature flows in Euclidean spaces.
We also derive a related  geometric  inequality
for  closed hypersurfaces with positive $k$-th mean curvature.
\end{abstract}

 \maketitle

 \markboth{Kwok-Kun Kwong and Pengzi Miao} {Monotone quantities  involving a weighted $\sigma_k$ integral}

\section{Introduction}
The aim of this paper is to introduce some  monotone quantities involving a weighted $ \sigma_k $ integral along inverse curvature flows in the Euclidean space $ \R^{m+1}$.

Given a  smooth closed hypersurface $ \Sigma \subset \R^{m+1} $,  let   $ \{ \kappa_1, \cdots, \kappa_m \}$  be
 the principal curvatures of $\Sigma$. For any $ 1 \le k \le m$,
define  the $k$-th mean curvature $H_k$ and the normalized $k$-th mean curvature $\sigma_k$ of $\Sigma$ by
$$ H_k = \sum_{ 1 \le i_1 < i_2 \ldots < i_k \le m }   \kappa_{i_1} \cdots \kappa_{i_k}  \ \ \
\textrm{and} \ \ \
\sigma_k= \frac{H_k}{{m\choose k}} . $$
respectively  (the notation $\sigma_k$ here follows that of  Reilly in \cite{reilly1973variational, reilly1977}).
When $ k = 0$,  define $H_0 = \sigma_0 = 1$.
A family of closed hypersurfaces $\{ \Sigma_t \}_{t \in I} $, given by a smooth map
$$ X: \Sigma \times I \longrightarrow \R^{m+1} $$
where $I$ is an open interval and  $ \Sigma_t = X ( \Sigma, t) $,
 is said to evolve according to an inverse curvature flow if
\be \label{eq-icf}
\frac{\p X}{\p t} = \frac{ \sigma_{k-1}}{\sigma_k} \nu
\ee
for some $ 1 \le k \le m $. Here $ \nu $ is the unit outward normal to $ \Sigma_t$ and $ \sigma_{k-1}, \sigma_{k}$ are computed on  $ \Sigma_t$.

Our main result is

\begin{theorem} \label{thm-monotone-1}
Suppose $ \{ \Sigma_t \}_{t \in I}$ is a smooth solution to an inverse curvature flow \eqref{eq-icf} for some $ 2 \le k \le m$.
Given any point $O \in \R^{m+1}$, let $ r $ be the Euclidean distance to $ O$. Then the function
\be \label{eq-qt-kge2-intro}
Q_k( \Sigma_t) = \lf( \int_{\Sigma_t} \sigma_{k-1}  d \mu \ri)^{ - \frac{ m -k}{m+ 1 -k} }  \lf(    \int_{\Sigma_t}  \sigma_k r^2 d \mu  -  \int_{\Sigma_t} \sigma_{k-2} d \mu \ri)
\ee
is monotone decreasing   and $Q_k(t)$ is a constant function if and only if $ \Sigma_t $ is a round sphere  for each $t$.
Here  $ d \mu $ is the volume form on $\Sigma_t$.
\end{theorem}

Here are some remarks concerning Theorem \ref{thm-monotone-1}.

\begin{remark}
The long time existence of smooth solutions to \eqref{eq-icf} was
established by  Gerhardt in \cite{gerhardt1990flow} and by Urbas in \cite{urbas1990expansion}
when the initial surface  $\Sigma$ is star-shaped with $\sigma_k > 0$.
Moreover,  they proved that the rescaled
hypersurfaces  $ \{ \widetilde{\Sigma_t } \}$, parametrized by $ \widetilde{X}(\cdot , t) =
  e^{- t } X( \cdot, t)   $, converge to a sphere  in the $ C^\infty$ topology as $ t \rightarrow \infty$.
\end{remark}

\begin{remark}
Theorem \ref{thm-monotone-1} does not address the case $k = 1$, i.e. when the flow is the inverse mean curvature flow.
In that case,  if one defines $\sigma_{-1} = \la Y, \nu \ra $ where $Y$ is a position vector field in $\R^{m+1}$,
 it was proved in \cite{KwongMiao} that
$$ {Q}_1( \Sigma_t ) = | \Sigma_t |^{ - \frac{ m -1}{m} }  \lf[ \frac{1}{m}  \int_{\Sigma_t}  H r^2 d \mu  - (m+1) \mathrm{Vol} (\Omega_t)   \ri] $$
is  monotone decreasing along the inverse mean curvature flow  $\frac{\p X}{\p t } = \frac{1}{H} \nu $.
Here $ |\Sigma_t|$ is the area of $\Sigma_t$,
$ \mathrm{Vol}(\Omega_t)$ is the volume of the region $ \Omega_t$ enclosed by $ \Sigma_t$ and
$H = \sum_{i=1}^m  \kappa_i $ is the mean curvature of $ \Sigma_t$.
Results  in  \cite{KwongMiao} were motivated by  the work of  Brendle, Hung and Wang in  \cite{Brendle-Hung-Wang}.
\end{remark}

In \cite{Guan-Li}, Guan and Li  proved that
\be \label{eq-monotone-GL}
  \lf( \int_{\Sigma_t} \sigma_{k-1} d \mu \ri)^{ - \frac{ m -  k}{ m + 1 -  k  } } \lf( \int_{\Sigma_t} \sigma_k d \mu \ri)
\ee
is monotone decreasing along the inverse curvature flow \eqref{eq-icf}. Using this together with the result of Gerhardt and Urbas, 
Guan and Li derived the quermassintegral inequalities 
\bee
  \lf( \frac{1}{\omega_m} {\int_{\Sigma} \sigma_{k-1} d \mu }\ri)^{ \frac{ 1}{ m + 1 - k } }  \le  \lf( \frac{1}{\omega_m}  {\int_{\Sigma} \sigma_k d \mu }\ri)^\frac{1}{m-k}
\eee
for any  star-shaped $\Sigma$ with $ \sigma_k > 0$. Here $\omega_m$ is the volume of the $m$-dimensional unit sphere in $\mathbb{R}^{m+1}$. 
Unlike the quantity in \eqref{eq-monotone-GL},   $Q_k(\Sigma_t)$ in \eqref{eq-qt-kge2-intro} is not scaling invariant (it has a unit of length square when scaled with respect to $O$). However, it is still interesting to know
if  $Q_k(\Sigma_t)  $ always has a fixed sign.
We answer this question in the next theorem.

\begin{theorem} \label{thm-ineq-weighted-intro}
Let $ \Sigma \subset \R^{m+1} $ be a smooth closed hypersurface. Suppose
 $\sigma_k > 0 $ on $\Sigma$  for some $ 1 \le k \le m $.
Given any point $O \in \R^{m+1}$,  let  $r$ be the Euclidean distance to $O$.
Then
\be \label{eq-ineq-weighted-intro}
\int_{\Sigma} \sigma_{l} r^p \ d \mu  \le  \int_{\Sigma}  \sigma_k   r^{ p + k - l}  d \mu
\ee
for any integer $ 0 \le l < k $ and any real number $p \ge 0$. Moreover,
the equality holds  if and only if $ \Sigma $ is a round sphere centered at $O$.
\end{theorem}
Let $l = k-2$ and $p=0$,  \eqref{eq-ineq-weighted-intro} becomes
\bee
\int_{\Sigma} \sigma_{k-2}  \ d \mu  \le  \int_{\Sigma}  \sigma_k   r^{2}  d \mu .
\eee
Therefore,  Theorem \ref{thm-ineq-weighted-intro} implies that $Q_k (t) $ is always nonnegative
 along the inverse curvature flow \eqref{eq-icf}.

This  paper is organized as follows. In Section \ref{sec-preliminary}, we
review some basics facts about $\sigma_k$. In Section \ref{sec-proof-thm-1},
we derive the monotonicity of $Q_k(\Sigma_t)$. In Section \ref{sec-proof-thm-2},
we prove the inequality \eqref{eq-ineq-weighted-intro}.

\section{Notations and preliminaries}  \label{sec-preliminary}
Given  a  smooth closed hypersurface $ \Sigma \subset  \R^{m+1}$,  let
$\overline \nabla $ and $\nabla $ denote  the connections on $\R^{m+1}$ and $\Sigma$ respectively.
Let $ \nu $ be the  outward unit normal   to $\Sigma$.
The shape operator of $ \Sigma$ with respect to $\nu$ is defined  by
 $$A (\cdot) =\overline \nabla_{(\cdot)}  \nu: T\Sigma \to T\Sigma ,$$
 where $ T \Sigma$ is the tangent bundle of $ \Sigma$.
 Given any local frame $\{ e_i\}_{i=1}^m $ on $ \Sigma$, define $\{ A_i^j \}$ by
 $ A (e_i) = A_i^j e_j $, where $ i, j  \in \{ 1, \ldots, m \}$ and Einstein summation convention is applied.
 Recall that  $ H_0 = \sigma_0 = 1 $ and
 \bee
 H_k = \sum_{ 1 \le i_1 < i_2 \ldots < i_k \le m }   \kappa_{i_1} \cdots \kappa_{i_k}  \ \ \
\textrm{and} \ \ \
\sigma_k= \frac{H_k}{{m\choose k}}
 \eee
 for $ 1 \le k \le m $,   where $ \{ \kappa_i \}_{i=1}^m $ are the eigenvalues of $ A$.
 In terms of $ \{ A_i^j\}$,  $H_k$ can be computed  by
\be \label{eq-formula-Hk}
\begin{split}
H_k = & \ \frac 1{k!}\sum_{\substack{1\le i_1,\cdots, i_k\le m\\
1\le j_1, \cdots, j_k\le m}} \delta_{j_1\cdots j_k}^{i_1\cdots i_k}A_{i_1}^{j_1}\cdots A_{i_k}^{j_k}  ,
\end{split}
\ee
where   $ \delta_{j_1\cdots j_k}^{i_1\cdots i_k} = 0 $ if   $i_p =i_q $ or $j_p =j_q$ for some $ p \ne q$, or if
 the two sets  $\lbrace i_1, \cdots, i_k\rbrace \ne \lbrace j_1, \cdots, j_k\rbrace$;
 otherwise $ \delta_{j_1\cdots j_k}^{i_1\cdots i_k} $
 is defined as the sign of the permutation $(i_1, \cdots, i_k)\mapsto (j_1, \cdots, j_k)$.
 Because of \eqref{eq-formula-Hk}, we define  $H_k = 0 $ for any $ k > m $.

A basic tool in the study of  $ H_k $  is  the $(k-1)$-th
Newton transformation $T_{k-1}: T\Sigma \rightarrow T \Sigma$ (cf. \cite{reilly1973variational, reilly1973hessian}).
If we write
$ T_{k-1} ( e_j ) = ( T_{k-1} )_j^i e_i $,
then $\{ (T_{k-1} )_j^i \} $ are given by
 \bee
  {(T_{k-1})}_j^{\,i}= \frac 1 {(k-1)!}
\sum_{\substack{1 \le i_1,\cdots, i_{k-1} \le m\\ 1\le j_1, \cdots, j_{k-1} \le m}}
\delta^{i  i_1 \ldots  i_{k-1} }_{j  j_1 \ldots  j_{k-1} }
A_{i_1}^{j_1}\cdots A_{i_{k-1}}^{j_{k-1}}.
\eee
A more geometric way to understand  $T_{k-1}$ is that
if $  \{ e_i \}_{i=1}^m $ consist of  eigenvectors of $A$ with
$ A (e_j) = \kappa_j e_j$, then $ T_{k-1} (e_j) =  \Lambda_j e_j  $,
where
\be \label{eq-Lambda-i}
 \Lambda_j =  \sum_{ \substack{ 1 \le i_1 < \cdots < i_{k-1} \le m, \\
j \notin \{ i_1, \cdots, i_{k-1 } \} } } \kappa_{i_1} \cdots \kappa_{i_{k-1}}  .
\ee
When $k=1$, one defines $T_0 = \mathrm{Id}$, the identity map.
It follows  from \eqref{eq-Lambda-i} that
\be \label{eq-trace-T}
\tr ( T_{k-1} ) =  [ m - (k-1) ] H_{k-1} ,
\ee
\be \label{eq-trace-TA}
\tr (T_{k-1} \circ A) = k H_k
\ee
and
\be \label{eq-trace-TA2}
\tr ( T_{k-1} \circ A \circ A) =  H_k H_1 - (k+1) H_{k+1}.
\ee
Another useful property of $T_{k-1}$ is that it is divergence free (cf. \cite{reilly1973hessian}),
\be  \label{eq-divergence-free}
\div \,T_{k-1} = 0 .
\ee
Here $ \tr (\cdot) $ and $\div (\cdot)$  denote the trace and the divergence taken on $\Sigma$ respectively.

Suppose $ \{\Sigma_t \}$ is a family of  evolving hypersurfaces
given by  a smooth map  $ X: \Sigma \times I \rightarrow \R^{m+1} $ with
\be \label{eq-F-flow}
\frac{\p X}{\p t} = F \nu
\ee
where $ \nu $ is the outward unit normal to $ \Sigma_t = X ( \Sigma, t)$ and $F$ denotes the speed of the flow
which may depend on $X$, the principal curvatures of $\Sigma_t$ and time $t$.
The following evolution equation of  $ H_k $ is standard (see, e.g.,  \cite[Proposition 4]{Guan-Li} or \cite[Lemma A]{reilly1973variational}),
\be \label{eq-flow-Hk}
\begin{split}
{H_k}' = & \  \tr ( T_{k-1} \circ A' ) \\
= & \   ( T_{k-1})_{j}^i
\lf[ -  (\nabla^2 F)_{i}^j  - F (A\circ A)^j_i \ri].
\end{split}
\ee
Here `` $ ' $ " denotes the derivative taken with respect to $t$.

\begin{prop} \label{prop-flow-Hf}
Let $ f$ be an arbitrary function on $ \R^{m+1}$.
Along the flow \eqref{eq-F-flow},
\be \label{eq-flow-Hf}
\begin{split}
\lf( \int_{\Sigma_t} H_l f d \mu \ri)'
= & \ \int_{\Sigma_t} \{  -    (T_{l-1} )_j^i    ( \bnabla^2 f )_i^j   + (l+1) H_l  \la \bnabla f, \nu \ra
 \\
& \ \ \ \ \ \    + (l+1)H_{l+1} f \} F d \mu
\end{split}
\ee
for any $ 1 \le l \le m$.
\end{prop}

\begin{proof}
Direct calculation gives
\be \label{eq-Direct-comp}
\begin{split}
\lf( \int_{\Sigma_t} H_l f d \mu \ri)' = & \
\int_{\Sigma_k} \lf( H_l' f + H_l \la \bnabla f, F \nu \ra + H_l f F H_1 \ri) d \mu \\
= & \ \int_{\Sigma_t} \lf\{  (T_{l-1})_j^i [ - (\nabla^2 F )_{i}^j - F ( A \circ A)_i^j ]  f  \ri. \\
& \ \ \ \ \ \ \ \lf. + H_l F  \la \bnabla f, \nu \ra + H_l H_1 F f \ri\} d \mu  \\
= & \ \int_{\Sigma_t} \lf\{  -(T_{l-1})_j^i   (\nabla^2 f )_{i}^j   +   (l+1)H_{l+1} f   \ri. \\
& \ \ \ \ \ \ \ \lf. + H_l   \la \bnabla f, \nu \ra \ri\} F d \mu  \\
\end{split}
\ee
where we used  \eqref{eq-flow-Hk}, \eqref{eq-divergence-free} and  \eqref{eq-trace-TA2}.

Note that the two  Hessians $ (\bnabla^2 f )_{ij} $  and $ (\nabla^2 f )_{ij} $  are related by
\be \label{eq-two-Hessians}
(\bnabla^2 f)_{ij}  = (\nabla^2 f)_{ij} + \frac{\p f}{\p \nu} A_{ij} ,
\ee
where $ A_{ij} = g_{il}A^l_j $ is the second fundamental form of $ \Sigma_t$ and $g$ denotes the
induced metric on $\Sigma_t$.

Therefore, it follows from \eqref{eq-Direct-comp} and \eqref{eq-two-Hessians} that
\bee
\begin{split}
\lf( \int_{\Sigma_t} H_l f d \mu \ri)'
= & \ \int_{\Sigma_t} \{  -    (T_{l-1} )_j^i    [  ( \bnabla^2 f )_i^j -    \la \bnabla f, \nu \ra  A_i^j  ]  \\
& \  \ \ \ \ \     + (l+1)H_{l+1} f  + H_l  \la \bnabla f, \nu \ra  \} F d \mu  \\
= & \ \int_{\Sigma_t} \{  -    (T_{l-1} )_j^i    ( \bnabla^2 f )_i^j   + (l+1) H_l  \la \bnabla f, \nu \ra  \\
& \ \ \ \ \ \    + (l+1)H_{l+1} f \} F d \mu  ,
\end{split}
\eee
where we also  used \eqref{eq-trace-TA}.
\end{proof}

We end this section by noting the following fact regarding  $\sigma_k  > 0$.

\begin{lemma} \label{lma-sigma-k}
For a closed hypersurface $ \Sigma $ in $\R^{m+1}$,
the condition $ \sigma_k > 0 $ implies $ \sigma_l > 0 , \ \forall \ 1 \le l \le k . $
\end{lemma}

This follows from  the characterization of
the Garding's cone
$$ \Gamma_k^+ = \{ \kappa \in \R^m \ | \ \sigma_{l} (\kappa) > 0 , \ \forall \ 1 \le  l \le k \} $$
as  the connected component of the set
$ \{ \kappa \in \R^m \ | \ \sigma_k (\kappa) > 0 \} $ containing the point $ (1, 1, \cdots, 1)$
(cf. \cite[Proposition 2.6]{Huisken-Sinestrari} or  \cite[Proposition 3.2]{Barbosa-Colares})
and  the fact that  $ \Sigma$ always has a point
 at which $ \kappa_i > 0 $, $\forall  \ 1 \le i \le m$.

\section{Derivation of monotone quantities}  \label{sec-proof-thm-1}

Given any point $O \in \R^{m+1}$, let $ r$ be the Euclidean distance to $O$.
Consider the function
$  f = \frac12 r^2 , $
which satisfies
\begin{itemize}
\item $ \bnabla^2 f = g $, where $ g$ is the Euclidean metric on $ \R^{m+1}$
\item  $  \bnabla f  =  Y $,
where $ Y $ is the position vector starting from $O$.
\end{itemize}
With such a choice of $f$, it follows from Proposition \ref{prop-flow-Hf}, \eqref{eq-trace-T} and the definition
$ \sigma_l = H_l / {m \choose l} $ that
\be \label{eq-flow-sigma-l}
\begin{split}
\lf( \int_{\Sigma_t}  \sigma_l f d \mu \ri)'
 = & \ \int_{\Sigma_t} \lf[  -   {l}  \sigma_{l-1}    + (l+1)  \sigma_l  \la Y, \nu \ra
  +  ( m - l )\sigma_{l+1} f \ri] F d \mu
\end{split}
\ee
 along the flow  \eqref{eq-F-flow},  $\forall \  1 \le l \le m $.

\begin{proof}[Proof of Theorem \ref{thm-monotone-1}]
Suppose the flow speed $F$ in \eqref{eq-F-flow}  is given by
$$ F = \frac{\sigma_{k-1}}{\sigma_k}  $$
for some $ 2 \le k \le m $.
Then \eqref{eq-flow-sigma-l} gives
\be \label{eq-icf-flow-sigma-l}
\begin{split}
\lf( \int_{\Sigma_t}  \sigma_l f d \mu \ri)'
 = & \ \int_{\Sigma_t} \lf[  -   {l}  \sigma_{l-1}    + (l+1)  \sigma_l  \la Y, \nu \ra
  +  ( m - l )\sigma_{l+1} f \ri] \frac{\sigma_{k-1}}{\sigma_k}   d \mu  .
\end{split}
\ee
Choose $l=k $ in \eqref{eq-icf-flow-sigma-l}, we have
\be \label{ineq-icf-flow-sigma-l=k}
\begin{split}
\lf( \int_{\Sigma_t}  \sigma_k f d \mu \ri)'
 = & \ \int_{\Sigma_t} \lf[  -   {k}  \frac{\sigma_{k-1}^2}{\sigma_k}    + (k+1)   \sigma_{k-1} \la Y, \nu \ra
  +  ( m - k )\sigma_{k+1} \frac{\sigma_{k-1}}{\sigma_k}  f \ri]   d \mu  \\
   \le  & \ \int_{\Sigma_t} \lf[  -   {k}  \sigma_{k-2}    + (k+1)   \sigma_{k-2}
  +  ( m - k )  {\sigma_k}  f \ri]   d \mu  \\
   =  & \ \int_{\Sigma_t}   \sigma_{k-2} \ d \mu +  (m-k) \int_{\Sigma_t} {\sigma_k}  f  d \mu  ,
\end{split}
\ee
where we  used the fact $ \sigma_k > 0 $,
 the Newton inequality
\bee
\sigma_{i-1} \sigma_{i+1} \le \sigma_i^2
\eee
for any $ 1 \le i \le m-1$,
and the Hsiung-Minkowski formula (\cite{hsiung1956some})
\bee
\int_\Sigma \sigma_{j-1} d \mu = \int_{\Sigma} \sigma_{j} \la Y, \nu \ra d \mu
\eee
for any $ 1 \le j \le m$.

We  need information on $ \lf( \int_{\Sigma_t} \sigma_{k-2} d \mu \ri)' $.
 Setting $ \displaystyle F = \frac{\sigma_{k-1}}{\sigma_k} $ and $f=1$ in Proposition \ref{prop-flow-Hf} gives
 \be \label{eq-icf-flow-H}
\lf( \int_{\Sigma_t} \sigma_l \  d \mu \ri)'
=  (m-l) \int_{\Sigma_t}  \sigma_{l+1}   \frac{\sigma_{k-1}}{\sigma_k}    d \mu .
\ee
Let $ l = k -2 $ in \eqref{eq-icf-flow-H}, we have
 \be \label{ineq-icf-flow-H-k-2}
 \begin{split}
\lf( \int_{\Sigma_t} \sigma_{k-2} \  d \mu \ri)'
= & \  [m-(k-2)] \int_{\Sigma_t}  \sigma_{k-1}   \frac{\sigma_{k-1}}{\sigma_k}    d \mu \\
\ge & \  [m-(k-2)] \int_{\Sigma_t}  {\sigma_{k-2}}   \ d \mu
\end{split}
\ee
where we again used the Newton inequality and the assumption $ k \ge 2 $.

Now it follows from \eqref{ineq-icf-flow-sigma-l=k} and \eqref{ineq-icf-flow-H-k-2} that
\bee
\begin{split}
 \lf[ \int_{\Sigma_t}  \lf( \sigma_k r^2   - \sigma_{k-2} \ri) d \mu \ri]' \leq
  (m-k) \lf[ \int_{\Sigma_t}  \lf( \sigma_k r^2   - \sigma_{k-2} \ri) d \mu \ri] ,
\end{split}
\eee
which then implies
\be \label{ineq-et-k-k-2}
\lf[ e^{- (m-k) t} \int_{\Sigma_t}  \lf( \sigma_k r^2   - \sigma_{k-2} \ri) d \mu \ \ri]' \le 0 .
\ee
On the other hand, setting $l=k-1$ in \eqref{eq-icf-flow-H} gives
\be \label{eq-k-1}
\lf( \int_{\Sigma_t} \sigma_{k-1}  \  d \mu \ri)'
=  [m- (k-1)] \int_{\Sigma_t}  \sigma_{k-1}   d \mu .
\ee
By \eqref{ineq-et-k-k-2} and \eqref{eq-k-1}, we conclude that
$$ \lf[ \lf( \int_{\Sigma_t} \sigma_{k-1}  \  d \mu \ri)^{ - \frac{ m-k}{ m - (k-1)} }  \int_{\Sigma_t}  \lf( \sigma_k r^2   - \sigma_{k-2} \ri) d \mu \ \ri]' \le 0 . $$
If the derivative is zero at some time $t_0$, then
$$ \kappa_1 = \cdots = \kappa_m  $$
at $t = t_0$ by the equality case in the Newton inequality, which  implies that $ \Sigma_{t_0}$ is a round sphere.
This completes the proof of Theorem \ref{thm-monotone-1}.
\end{proof}

\begin{remark}
Let $l = k $ in \eqref{eq-icf-flow-H} and apply the Newton inequality, one has
 \be \label{eq-icf-flow-H-k}
 \begin{split}
\lf( \int_{\Sigma_t} \sigma_k \  d \mu \ri)'
\le & \  (m-k) \int_{\Sigma_t} {\sigma_k}    d \mu .
\end{split}
\ee
\eqref{eq-icf-flow-H-k} and  \eqref{eq-k-1} imply  the quantity in \eqref{eq-monotone-GL}
is monotone decreasing, which is the monotonicity of Guan and Li (\cite{Guan-Li}).
\end{remark}

\begin{remark}
In \eqref{ineq-icf-flow-H-k-2},  if $k=1$, we do not have a point-wise inequality of the form
 $ \frac{1}{\sigma_1} = \frac{m}{H} \ge \sigma_{-1} = \la Y, \nu \ra $.
In this case,  an analogue of \eqref{ineq-icf-flow-H-k-2} in \cite{KwongMiao} was
 \bee
 \begin{split}
 \mathrm{Vol} (\Omega_t)'
\ge & \  (m+1) \mathrm{Vol} (\Omega_t)
\end{split}
\eee
which was derived using an inequality of Ros (\cite{ros1987compact})
\bee
 m \int_{\Sigma} \frac{1}{H} d \mu \ge (m+1) \mathrm{Vol} (\Omega) .
\eee
\end{remark}

\section{A related inequality} \label{sec-proof-thm-2}
In this section, we  prove Theorem \ref{thm-ineq-weighted-intro}.
First, we need  a generalized Hsiung-Minkowski formula (cf. \cite[Theorem 2.1]{kwong2012}).
Similar formulas of this type can  be found in \cite{Alencar-Colares, Chen-Yano, reilly1973variational, WS}.

\begin{prop} \label{prop-generalized-HM}
Let $\Sigma$ be a smooth closed hypersurface in $\mathbb{R}^{m+1}$ and
$f$  be a smooth function on $\Sigma$.
Then
  \be \label{eq-generalized-HM}
  \begin{split}
    \int_\Sigma   f \sigma_{l} d\mu = & \ \int_\Sigma f\sigma_{l+1}  \la Y, \nu \ra  d\mu \\
    & \ - \frac{1}{ ( m - l )  {m \choose l  } }\int_\Sigma \langle T_{l}(\nabla f), Y^\parallel \rangle d\mu ,
  \end{split}
  \ee
  for any $ 0 \le l < m $.
    Here $Y$ is  a position vector field, $Y^\parallel $ denotes its tangential component  on $\Sigma$,
    and $\nu$ is the unit outward normal to $\Sigma$.
\end{prop}
\begin{proof}
$ Y$ being  a  position vector field  implies
 \be \label{eq-nabla-Y-parallel}
 \begin{split}
 ( \nabla Y^\parallel )^j_i =  ( \bnabla Y)^j_i - A^j_i \la Y, \nu \ra  = \delta^j_i - A^j_i \la Y, \nu \ra .
 \end{split}
 \ee
Therefore, on $ \Sigma$
\be \label{eq-div-1}
\begin{split}
\div (  f T_{l} ( Y^\parallel )  = & \  \la \nabla f , T_{l} ( Y^\parallel) \ra + f    (T_{l})_j^i ( \nabla Y^\parallel )^j_i   \\
  = & \ \la  T_{l} (\nabla f ), Y^\parallel  \ra   + f    [  \tr (T_{l}) -  \tr ( T_{l} \circ A ) \la Y, \nu \ra  ]  ,
\end{split}
\ee
where we  used  \eqref{eq-divergence-free}, \eqref{eq-nabla-Y-parallel} and  the fact
$ T_{l} $ is self-adjoint.
It follows from \eqref{eq-trace-T}, \eqref{eq-trace-TA} and \eqref{eq-div-1} that
\be \label{eq-div-2}
\begin{split}
\div (  f T_{l} ( Y^\parallel )
  = & \ \la  T_{l} (\nabla f ), Y^\parallel  \ra   + f    [  ( m - l ) H_{l}  -   (l+1) H_{l+1}  \la Y, \nu \ra  ]  .
\end{split}
\ee
Integrating \eqref{eq-div-2} over $\Sigma$ gives
\eqref{eq-generalized-HM}.
\end{proof}

Next, we need a result concerning the positivity of the Newton transformation $T_{l}$ in \cite[Proposition 3.2]{Barbosa-Colares}.

\begin{prop}[\cite{Barbosa-Colares}] \label{prop-BC}
For a closed hypersurface $ \Sigma \subset \R^{m+1} $, if  $\sigma_k > 0 $ for some $ 1 \le k \le m$, then
the quadratic form associated to $ T_{l} $ is positive definite for any $ 0 \le l < k$.
\end{prop}

\begin{proof}[Proof of Theorem \ref{thm-ineq-weighted-intro}]
We first assume $ O \notin \Sigma$. In this case, $ r $ is a smooth positive function when restricted to $ \Sigma$.
Choose $ f = r^p  $ in Proposition \ref{prop-generalized-HM}, we have
  \be \label{eq-generalized-HM-rp}
  \begin{split}
    \int_\Sigma   r^p \sigma_{l} d\mu = & \ \int_\Sigma r^p \sigma_{l+1}  \la Y, \nu \ra  d\mu \\
    & \ - \frac{1}{ ( m - l )  {m \choose l  } }\int_\Sigma \langle T_{l}(\nabla r^p), Y^\parallel \rangle d\mu ,
  \end{split}
  \ee
where
\be \label{eq-nabla-rp}
\nabla r^p = p r^{p-1} \nabla r = p r^{p-2} { Y^\parallel}.
\ee
By Proposition \ref{prop-BC},
\be \label{eq-good-T}
\la T_{l} (Y^\parallel ), Y^\parallel \ra \ge 0
\ee
and $\la T_{l} (Y^\parallel ), Y^\parallel \ra =  0$ if and only if $Y^\parallel = 0$.
Thus it follows from \eqref{eq-generalized-HM-rp}, \eqref{eq-nabla-rp}, \eqref{eq-good-T} and the assumption $p \ge 0 $
that
  \be \label{eq-generalized-HM-rp-1}
    \int_\Sigma    \sigma_{l} r^p d\mu \le  \int_\Sigma  \sigma_{l+1} r^{p+1}  d\mu .
  \ee
If $ l + 1 < k $,  applying  \eqref{eq-generalized-HM-rp-1} repeatedly with $(l, p)$ replaced by $(l+1, p+1)$, $\ldots$, $( k - 1, p + k-l-1)$ respectively gives
\bee
  \int_\Sigma  \sigma_{l+1} r^{p+1}  d\mu \le \int_\Sigma  \sigma_{l+2} r^{p+2}  d\mu \le \cdots
 \le \int_\Sigma \sigma_{k} r^{ p + k - l } d \mu .
\eee
This  proves the inequality \eqref{eq-ineq-weighted-intro}.  When  the equality in \eqref{eq-ineq-weighted-intro} holds, the equality in
\eqref{eq-generalized-HM-rp-1} must hold, which
implies $Y^\parallel = 0 $. Hence $\Sigma$ is a round sphere centered at $O$.

Next suppose  $ O \in \Sigma$. Let $ B_\epsilon \subset \Sigma$ be a small geodesic ball with geodesic radius $\epsilon$.
Integrating \eqref{eq-div-2} over $ \Sigma \setminus B_\epsilon$ with $f= r^p$, we have
 \be \label{eq-generalized-HM-singular}
  \begin{split}
    \int_{\Sigma \setminus B_\epsilon}  r^p \sigma_{l} d\mu = & \ \int_{\p B_\epsilon} r^p \la T_{l} (Y^\parallel) , {n} \ra d \tau + \int_{\Sigma \setminus B_\epsilon}  r^p \sigma_{l+1}  \la Y, \nu \ra  d\mu \\
    & \ - \frac{1}{ ( m - l )  {m \choose l  } }\int_{\Sigma \setminus B_\epsilon}  \langle T_{l}(\nabla r^p), Y^\parallel \rangle d\mu ,
  \end{split}
  \ee
where $ n$ is the inward unit normal to $ \p B_\epsilon $ in $\Sigma$ and $d \tau $ is the volume form on $\p B_\epsilon$.
It is clear that
$$ \int_{\p B_\epsilon} r^p \la T_{l} (Y^\parallel) , {n} \ra d \tau \rightarrow 0, \ \mathrm{as} \ \epsilon \rightarrow 0 . $$
Therefore  \eqref{eq-generalized-HM-singular} implies
 \bee \label{eq-generalized-HM-minus-a-point}
  \begin{split}
    \int_{\Sigma }  r^p \sigma_{l} d\mu = & \  \int_{\Sigma }  r^p \sigma_{l+1}  \la Y, \nu \ra  d\mu
     - \frac{1}{ ( m - l )  {m \choose l  } }\int_{\Sigma \setminus \{ O \} }  \langle T_{l}(\nabla r^p), Y^\parallel \rangle d\mu ,
  \end{split}
  \eee
from which the inequality  \eqref{eq-ineq-weighted-intro} and its equality case  follow as in the previous case.
\end{proof}

\end{document}